\documentclass[11pt]{article}

\usepackage{amsmath}
\usepackage{amsfonts}
\usepackage{amssymb}
\usepackage{amsthm}
\usepackage{graphics}
\usepackage{amscd}
\usepackage{graphicx,epsfig}
\usepackage{color}
\usepackage[all]{xy}
\usepackage{url}

\DeclareMathOperator{\Div}{div}

\renewcommand{\epsilon}{\varepsilon}
\newcommand{\R}{\mathbb{R}}

\newcommand{\C}{\mathbb{C}}

\renewcommand{\S}{\mathbb{S}}
\newcommand{\N}{\mathbb{N}}

\newcommand{\eps}{\varepsilon}
\newcommand{\id}{\text{id}}

\newcommand{\Ome}{\Omega}

\newtheorem{thm}{Theorem}
\newtheorem{cor}[thm]{Corollary}
\newtheorem{prop}[thm]{Proposition}
\newtheorem{lem}[thm]{Lemma}

\renewcommand{\phi}{\varphi}

\newtheorem*{thm*}{Theorem}

\newtheorem*{claim*}{Claim}

\theoremstyle{remark}

\newtheorem*{rem*}{Remark}

\newcounter{remark}

\newcounter{case}

\newcounter{construction}

\newcounter{fact}

\newcounter{step}

\newcommand{\Onp}{O_{p+1}(\R)\times O_{n-p+1}(\R)}

\theoremstyle{plain}

\newtheorem*{thmmain}{Theorem~\ref{th:main}}
\newtheorem*{cormain}{Corollary~\ref{cor}}

\title{Minimal hypersurfaces asymptotic to Simons cones}
\author{Laurent Mazet \thanks{The author was partially supported by the
ANR-11-IS01-0002 grant.}}
\date{}

\begin{document}

\maketitle

\begin{abstract}
In this paper, we prove that, up to similarity, there are only two minimal
hypersurfaces in $\R^{n+2}$ that are asymptotic to a Simons cone, \textit{i.e.}
the minimal cone over the minimal hypersurface $\sqrt{\frac pn}\S^p\times
\sqrt{\frac{n-p}n} \S^{n-p}$ of $\S^{n+1}$
\end{abstract}




\section{Introduction}


One important property of minimal hypersurfaces is the monotonicity formula.
If $\Sigma$ is a proper minimal hypersurface in $\R^{n+2}$, it says that the
quantity
$$
\theta(p,r)=\frac1{\omega_{n+1}r^{n+1}}\textrm{Vol}(\Sigma\cap B(p,r))
$$
is a non decreasing function of $r$ (here $\omega_{n+1}$ is the volume of the
unit ball of dimension $n+1$ and $B(p,r)$ denote the ball of $\R^{n+2}$ centered
at $p$ and radius $r$). We notice that in this paper, by minimal hypersurface, we mean
smooth proper hypersurface with vanishing mean curvature.

Hence we can define the density at infinity of $\Sigma$ as
$\theta_\infty(\Sigma)=\lim_\infty \theta(p,r)$. The monotonicity implies that
$\theta_\infty(\Sigma)\ge 1$ and $\theta_\infty(\Sigma)= 1$ iff $\Sigma$ is a
hyperplane. It also implies that if $\theta_\infty(\Sigma)\le2$, $\Sigma$ is embedded. One
interesting question is to understand the gap between this
value $1$ and the density at infinity of $\Sigma$ for $\Sigma$ not a hyperplane.

When $\theta_\infty(\Sigma)$ is finite, the asymptotic behaviour of $\Sigma$ is
given by a minimal cone which is the limit of a blow-down sequence
$(t_i\Sigma)_{i\in \N}$ with $t_i\searrow 0$ (here the limit is in the varifold
sense). This cone has density
$\theta_\infty(\Sigma)$ so the study of minimal cones is important to understand
what are the possible densities at infinity.

In dimension $3$ ($n=1$), it is known that $\theta_\infty(\Sigma)\ge 2$ and it
is conjectured that this value $2$ is only realized by catenoids and singly
periodic Scherk surfaces (see \cite{MeWo} for a partial answer by Meeks and
Wolf). In dimension
$4$ ($n=2$), the proof of the Willmore conjecture by Marques and Neves
\cite{MaNe} implies that $\theta_\infty(\Sigma)\ge \pi/2$ if $\Sigma$ is non
planar and this value
corresponds to the cone over a Clifford torus. In higher dimension, good
candidates for the lowest value of the density at infinity are the one of the cone
over product of spheres. More precisely the submanifold
$$
S_{n,p}=\sqrt{\frac pn}\S^p\times \sqrt{\frac{n-p}n} \S^{n-p}
$$
is a minimal hypersurface of $\S^{n+1}$ (notice that $S_{2,1}$ is a Clifford
torus). The cone $C_{n,p}$ over $S_{n,p}$ is a
good candidate for the lowest density at infinity; such a cone (and its image by
linear isometry) is called a Simons cone. More precisely, if $n$ is even
it is conjectured that the density of $C_{n,n/2}$ is a lower bound for the
density at infinity of a non planar minimal hypersurface of $\R^{n+2}$
and, if $n$ is odd, the lower bound is given by $C_{n,(n-1)/2}$. The best known
result about that question is given by Ilmanen and White in \cite{IlWh}; they
obtain lower bounds for the density of some area-minimizing cones under
topological assumptions.

The aim of this paper is to understand the minimal hypersurfaces whose
asymptotic behaviour is given by $C_{n,p}$. This cone is invariant by the subgroup
$O_{p+1}(\R)\times O_{n-p+1}(\R)$ of $O_{n+2}(\R)$. Actually we are going to prove that a
minimal hypersurface asymptotic to $C_{n,p}$ is also invariant by this subgroup.
This implies that, up to homotheties and translations, there are only two such
hypersurfaces. So our main result can be stated as follows.

\begin{thm}
\label{th:main}
For any $n\ge 2$ and $1\le p\le n-1$ there are two minimal hypersurfaces
$\Sigma_{n,p,\pm}\subset\R^{n+2}$ such the following is true. If $\Sigma$ is a
minimal hypersurface of $\R^{n+2}$ asymptotic to a Simons cone, then
$\Sigma=f(\Sigma_{n,p,\pm})$ for some $p\in\{1,\dots,n-1\}$, sign $\pm$ and a
similarity $f$. Moreover $\Sigma_{2p,p,+}=\Sigma_{2p,p,-}$.
\end{thm}

After writing the paper, the author has discovered that the same question was
studied by Simon and Solomon in \cite{SiSo}. They got the same result but
with the restriction that the cone $C_{n,p}$ is area minimizing that is $n\ge 6$
and, if $n=6$, $p\notin\{1,5\}$. So Theorem~\ref{th:main} generalizes their
result to any value of $n$ and $p$.

In dimension $4$ ($n=2$), the proof of the Willmore conjecture by Marques and
Neves \cite{MaNe} gives the following corollary which identifies the non planar
minimal hypersurfaces with the lowest density at infinity.

\begin{cor}\label{cor}
Let $\Sigma$ be a minimal hypersurface of $\R^4$ with
$\theta_\infty(\Sigma)=\frac\pi2$ then $\Sigma=f(\Sigma_{2,1,\pm})$ for a
similarity $f$.
\end{cor}

The proof of the main theorem starts with a result of Allard and Almgren
\cite{AlAl}, which implies that a minimal hypersurface $\Sigma$ asymptotic to
$C_{n,p}$ can be described as a normal graph over $C_{n,p}$ and the function
defining the graph
has a certain asymptotic. The first part of the proof consists in improving this
asymptotic to get a very good description of the behaviour of $\Sigma$ outside a
compact subset. A similar work appears in the paper of Simon and Solomon
\cite{SiSo} but we add some extra arguments to deal with low values of $n$.

Using this description, we are then able to apply the Alexandrov reflection
technique \cite{Ale} to $\Sigma$ to prove that it possesses a lot of symmetries
and then is invariant by $\Onp$. Here we apply this technique to non compact
hypersurfaces; this is why we need to know the asymptotic behaviour of $\Sigma$
(see \cite{Sch2} for a similar situation). This argument is different from the
one of Simon and Solomon. The last step of the proof consists
in classifying the minimal hypersurfaces invariant by such a group of
isometries.

The paper is divided as follow. In Section~\ref{sec:prel}, we recall some
definitions and study the Simons cones $C_{n,p}$. We mainly study the minimal
surface equation satisfied by normal graphs over $C_{n,p}$. We are interested on
the asymptotic behaviour of solutions of this equation. In
Section~\ref{sec:sym}, we prove that a minimal hypersurface asymptotic to $C_{n,p}$
is $\Onp$ invariant. This is the main step of the proof of
Theorem~\ref{th:main}; we also give the proof of Corollary~\ref{cor}. In
Section~\ref{sec:ode}, we classify all minimal hypersurfaces that are
invariant by $\Onp$. The paper ends with two appendices where we give two results
used in Sections~\ref{sec:prel} and \ref{sec:sym}.

\textbf{Acknowledgments.} The author would like to thank Fernando Marques for discussions which are at the
origin of this work. He would like also to thank the referee for all the precise and
important remarks he made about the writing of the paper.


\section{Preliminar results}
\label{sec:prel}


\subsection{Density at infinity and asymptotic behaviour of minimal hypersurfaces}
\label{sec:density}


Let $\Sigma$ be a proper smooth minimal hypersurface in $\R^{n+2}$. The monotonicity
formula tells that the quantity.
$$
\theta(p,r)=\frac1{\omega_{n+1}r^{n+1}}\textrm{Vol}(\Sigma\cap B(p,r))
$$
is increasing in $r$; here $\omega_{n+1}$ is the volume of the unit ball of
dimension $n+1$ and $B(p,r)$ denote the ball of $\R^{n+2}$ centered at $p$ and
radius $r$.

Hence we can define the density at infinity of $\Sigma$ as
$\theta_\infty(\Sigma)=\lim_\infty \theta(p,r)$. This definition does not depend
on the point $p$. Choosing $p\in \Sigma$, we get $\theta_\infty(\Sigma)\ge
\lim_0 \theta(p,r)=1$ and the equality case in the monotonicity formula says
that  $\theta_\infty(\Sigma)=1$ if and only if $\Sigma$ is planar.

Assume now that $\theta_\infty(\Sigma)$ is finite. Then if $(t_i)_{i\in\N}$ is a
decreasing sequence converging to $0$, there is a subsequence such that the
blow-down sequence $(t_i\Sigma)$ converges to $C$ in the varifold sense where
$C$ is a cone over a stationary varifold of $\S^{n+1}$. This cone $C$ is called
a limit cone of $\Sigma$, we also say that $\Sigma$ is asymptotic to $C$. \textit{A priori}, the cone $C$ depends on the chosen
sequence
$(t_i)$ and is not smooth outside the origin. As an example, if $\Sigma$ is a
catenoid, $\theta_\infty(\Sigma)=2$ and $C$ is a plane with multiplicity $2$.
Notice that in dimension $3$ ($n=1$), except the plane, no minimal cone is
smooth outside the origin.

In $\R^4$ ($n=2$), if $\theta_\infty(\Sigma)<2$; the proof of Theorem A.1 in
\cite{MaNe} given by Marques and Neves implies that a limit cone $C$ is smooth
outside the origin. So $C$
is a cone over a smooth minimal surface $S$ of $\S^3$. If
$\theta_\infty(\Sigma)>1$, $S$ is not an equator of $\S^3$ and has non zero
genus (see Almgren~\cite{Alm2}). So Theorem B in \cite{MaNe} implies that the area of $S$
is at least
$2\pi^2$ and is $2\pi^2$ if and only if $S$ is a Clifford torus. Thus
$\theta_\infty(\Sigma)\ge \frac\pi2$ and $\theta_\infty(\Sigma)= \frac\pi2$ iff
a limit cone $C$ is a cone over a Clifford torus (in that case $C$ does not
depend on the blow-down sequence by a result of Allard and Almgren \cite{AlAl}).


\subsection{The Simons cones}


The aim of this paper is to identify a minimal hypersurface in terms of its
limit cone. Actually we are interested to particular minimal cones.

For $n\ge 2$ and $1\le p\le n-1$, let us write $\R^{n+2}=
\R^{p+1}\times\R^{n-p+1}$ and consider the submanifold $S_{n,p}=\sqrt{\frac
pn}\S^p\times \sqrt{\frac{n-p}n} \S^{n-p}$ which is a minimal hypersurface of
$\S^{n+1}$. Let $C_{n,p}$ be the minimal cone over the minimal hypersurfaces
$S_{n,p}$. The surface $S_{2,1}$ is the Clifford torus of $\S^3$ and
$C_{2p,p}$ is the classical Simons cone~\cite{Sims}. So in the following, we call $C_{n,p}$
a \textit{Simons cone}. Actually, any image of $C_{n,p}$ by a linear
isometry is also called a Simons cone.

$C_{n,p}$ can be parametrized by
$$
X:\R\times\S^p\times\S^{n-p}\rightarrow \R^{n+2}; (t,x,y)\mapsto e^t(\sqrt{\frac
pn}x,\sqrt{\frac{n-p}n}y).
$$
Using this coordinate system, the metric on $C_{n,p}$ is $e^{2t}(dt^2+\frac pn
ds_1^2+\frac{n-p}n ds_2^2)$ where $ds_1^2$ and $ds_2^2$ are respectively the
round metrics on $\S^p$ and $\S^{n-p}$. The unit normal vector to $C_{n,p}$ is given by
$$
N(t,x,y)=(\sqrt{\frac{n-p}n} x,-\sqrt{\frac pn}y)
$$
Let $(e_i)$ and $(f_\alpha)$ be respectively orthonormal bases of $T_x\S^p$ and
$T_y\S^{n-p}$. Then an orthonormal basis of $T_{X(t,x,y)}C_{n,p}$ is given by 
$$
((\sqrt{\frac
pn}x,\sqrt{\frac{n-p}n}y),(e_1,0),\cdots,(e_p,0),(0,f_1),\cdots,(0,f_{n-p}))
$$
In this basis, the shape operator $S$ of $C_{n,p}$ is diagonal with 
\begin{align*}
S&((\sqrt{\frac
pn}x,\sqrt{\frac{n-p}n}y))=0\\
S&((e_i,0))=e^{-t}\sqrt{\frac{n-p}p}(e_i,0)\\
S&((0,f_\alpha))=-e^{-t}\sqrt{\frac p{n-p}}(0,f_\alpha)
\end{align*}


\subsection{The minimal surface equation}


In the following of the paper, we study minimal hypersurfaces of $\R^{n+2}$ that
can be described as normal graphs over a cone $C_{n,p}$. More precisely, such a
surface is the image of the following parametrization:
$$
Y: (t,x,y)\mapsto e^t\Big((\sqrt{\frac pn}x,\sqrt{\frac{n-p}n}y)+g(t,x,y)
(\sqrt{\frac{n-p}n}x,-\sqrt{\frac pn}y)\Big)
$$
where $g$ is a smooth function defined on a domain of
$\R\times\S^p\times\S^{n-p}$.

Using computations of the preceding section, this hypersurface is minimal if $g$
satisfies to the following partial differential equation:
\begin{equation}\label{eq:mse1}
\begin{split}
0&=\partial_t\left(\frac{g+g_t}W \right)+ \frac
n{p(1+\sqrt{\frac{n-p}p}g)}\Div_1(\frac{\nabla^1
g}{(1+\sqrt{\frac{n-p}p}g)W})\\
&\quad\quad+\frac n{(n-p)(1-\sqrt{\frac p{n-p}}g)}
\Div_2(\frac{\nabla^2 g}{(1-\sqrt{\frac p{n-p}}g)W})\\
&\quad\quad+\frac {ng+(g+g_t)(n+\frac{n(n-2p)}{\sqrt{p(n-p)}}
 g)}{W(1+\frac{n-2p}{\sqrt{p(n-p)}}g-g^2)}
\end{split}
\end{equation}
where $\nabla^1$, $\nabla^2$, $\Div_1$, $\Div_2$ are respectively the gradient
and the divergence operator for the round metric with respect to the $x\in\S^p$
and $y\in\S^{n-p}$ variables and $W$ is given by the following expression:
$$
W=\Big(1+(g+g_t)^2+\frac np\frac{|\nabla^1
g|^2}{(1+\sqrt{\frac{n-p}p}g)^2}+ \frac n{n-p}
\frac{|\nabla^2g|^2}{(1-\sqrt{\frac p{n-p}}g)^2}\Big)^{\frac12}
$$

The expression of Equation \eqref{eq:mse1} is long but we notice that it is an
elliptic second order equation and moreover it is uniformly elliptic if $\nabla
g$ is uniformly bounded.

Besides, for most of our arguments, we only need a simplified version of Equation
\eqref{eq:mse1}. Indeed the function $g$ will be close to $0$, so we will
use the following form:
\begin{equation}\label{eq:mse2}
0=g_{tt}+\frac np \Delta_1 g+\frac n{n-p}\Delta_2 g+(n+1) g_t+2ng+ Q(g)
\end{equation}
where $\Delta_1$ and $\Delta_2$ are respectively the Laplace operator with
respect to $x$ and $y$ variables and $Q(g)$ gathers all the nonlinear terms of
Equation~\eqref{eq:mse1}.


\subsection{The kernel of the linearized operator}


The linearized operator of the minimal surface equation \eqref{eq:mse2} is 
$$
Lu=u_{tt}+\frac np \Delta_1 u+\frac n{n-p}\Delta_2 u+(n+1) u_t+2nu
$$

Our analysis of solutions of \eqref{eq:mse1} is based on the asymptotic behaviour
of elements in the kernel of $L$. Such an element in the kernel can be decomposed as
the sum of terms of the form $v(t)\Phi(x)\Psi(y)$ where $\Phi$ and $\Psi$ are
respectively eigenfunctions of the Laplace operator on $\S^p$ and $\S^{n-p}$.
The eigenvalues of $\Delta$ on $\S^m$ are $-k(k+m-1)$ ($k\ge 0$). So
$(t,x,y)\mapsto v(t)\Phi(x)\Psi(y)$ is in the kernel if $v$ satisfies the
following ode for some $k$ and $l$:
$$
0=v_{tt}+(n+1)v_t+(2n-\frac np k(k+p-1)-\frac n{n-p} l(l+(n-p)-1))v
$$
The asymptotic behaviour of $v$ is given by the roots of 
$$
0=\lambda^2+(n+1)\lambda+2n-\frac np k(k+p-1)-\frac n{n-p} l(l+(n-p)-1)
$$
In the following, these roots are denoted by $\lambda_{k,l,\pm}$.
Actually, we are only interested in roots whose real part is between $-2$ and
$0$.

If $k+l=0$, the equation is $0=\lambda^2+(n+1)\lambda+2n$ whose
discriminant $(n+1)^2-8n$ is negative if $n<6$ and positive if $n\ge 6$. So the
roots are 
$$
\begin{cases}
\lambda_{0,0,\pm}&=\frac{-(n+1)\pm i\sqrt{8n-(n+1)^2}}2 \text{ if }n<6\\
\lambda_{0,0,\pm}&=\frac{-(n+1)\pm \sqrt{(n+1)^2-8n}}2 \text{ if }n\ge6
\end{cases}
$$
If $n\ge 6$, a computation gives $\lambda_{0,0,-}<\lambda_{0,0,+}<-2$. So the
real part of $\lambda_{0,0,\pm}$ is between $-2$ and $0$ only for $n=2,3$.

If $k+l=1$, the equation is $0=\lambda^2+(n+1)\lambda+n=(\lambda+1)
(\lambda+n)$. So $\lambda_{k,l,+}=-1$ and $\lambda_{k,l,-}=-n$ which lies in
$[-2,0)$ if $n=2$.

If $k+l\ge 2$, $2n-\frac np k(k+p-1)-\frac n{n-p} l(l+(n-p)-1)\le 0$, so
$\lambda_{k,l,+}\ge 0$ and $\lambda_{k,l,-}\le -(n+1)\le -3$.


\subsection{Asymptotic behaviour of a minimal graph}


In this section, we study the asymptotic behaviour of a minimal normal graph
over a cone $C_{n,p}$. Actually, we prove an improvement result for the
asymptotic behaviour of solutions of \eqref{eq:mse1}.

First we recall a classical definition of weighted norm for functions on
$\R_+\times \S^p\times\S^{n-p}$. If $u$ is a continuous function on $\R_+\times
\S^p\times\S^{n-p}$ and $\delta$ is a real number, we define its weighted norm
$$
\|u\|_\delta=\sup \{e^{\delta t}|u(t,x,y)|, (t,x,y)\in\R_+\times\S^p\times
\S^{n-p}\}
$$
when this quantity is finite. When $\|u\|_\delta<+\infty$, we will also write
$u=O(e^{-\delta t})$.

We then have the following result that describes the asymptotic behaviour of a
solution of \eqref{eq:mse1} with $\|u\|_\delta$ finite for $\delta>0$.

\begin{prop}\label{prop:improv}
Let $u$ be a solution of \eqref{eq:mse1} on $\R_+\times \S^p\times\S^{n-p}$ such
that $\nabla u$ is uniformly bounded and $\|u\|_\delta<+\infty$ with $\delta>0$
and $-2\delta\neq \lambda_{k,l,\pm}$ for all $k,l\ge 0$. Then $u$ can be written
$u=v+r$ where $\|v\|_\delta<+\infty$ satisfies to $L(v)=0$ and
$\|r\|_{2\delta}<+\infty$.
\end{prop}

\begin{proof}
The proof is based on the spectral decomposition of functions on
$\S^p\times\S^{n-p}$. 

First, since $\nabla u$ is uniformly bounded, Equation \eqref{eq:mse1} is
uniformly elliptic, so classical elliptic estimates give upper bounds on the
derivatives of $u$: more precisely, for any $m>0$, there is a constant $C_m$
such that for any $s>1$
$$
\|\nabla^m u \|_{C^0([s,s+1]\times\S^p\times \S^{n-p})}\le C_m\|u\|_{C^0(
[s-1,s+2]\times\S^p\times \S^{n-p})}
$$
This implies that for any $m$, $\|\nabla^m u \|_\delta<+\infty$. Since the term
$Q(u)$ in \eqref{eq:mse2} gathers all the nonlinear terms in $u$ we have
$\|Q(u)\|_{2\delta}<\infty$ and $\|\nabla^m Q(u)\|_{2\delta}<\infty$. 

In the preceding section, we have describe the spectrum of the Laplace operator
on the
sphere. So let us denote $\lambda_k=k(k+p-1)$ and $\Phi_{k,\alpha}$ an 
orthonormal basis of the eigenspace of $\Delta_1$ associated to $-\lambda_k$ on
$\S^p$. We also denote $\mu_l=l(l+n-p-1)$ and $\Psi_{l,\beta}$ an orthonormal
basis of the eigenspace of $\Delta_2$ associated to $-\mu_l$ on $\S^{n-p}$. The
multplicity of the $-\lambda_k$ and $-\mu_l$ are respectively bounded by $c
(k^p+1)$ and $c(l^{n-p}+1)$. Moreover, we have the following estimates for the
$L^\infty$ norm of the eigenfunctions (see \cite{Sog}):
$$
\|\Phi_{k,\alpha}\|_\infty\le c \lambda_k^{\frac{p-1}4}\textrm{ and
}\|\Psi_{l,\beta}\|_\infty\le c \mu_l^{\frac{(n-p)-1}4}.
$$

Now let us define 
\begin{align*}
g_{k,l,\alpha,\beta}(t)&=\int_{\S^p\times \S^{n-p}} u(t,x,y) \Phi_{k,\alpha}(x)
\Psi_{l,\beta}(y) dx dy,\\
f_{k,l,\alpha,\beta}(t)&=-\int_{\S^p\times \S^{n-p}} Q(u)(t,x,y)
\Phi_{k,\alpha}(x) \Psi_{l,\beta}(y) dx dy.
\end{align*}

$g_{k,l,\alpha,\beta}$ and $f_{k,l,\alpha,\beta}$ are smooth functions on $\R_+$
and, from \eqref{eq:mse2}, they satisfy
$$
g_{k,l,\alpha,\beta}''-(\lambda_{k,l,+}+\lambda_{k,l,-})g_{k,l,\alpha,\beta}'+
(\lambda_{k,l,+}\times\lambda_{k,l,-})g_{k,l,\alpha,\beta}= f_{k,l,\alpha,\beta}
$$

Using $\Delta_1 \Phi_{k,\alpha}=-\lambda_k\Phi_{k,\alpha}$, $\Delta_2
\Psi_{l,\beta}=-\mu_l\Psi_{l, \beta}$ and integration by parts, we get the
following estimates for $a,b\in\N$:
\begin{align*}
|g_{k,l,\alpha,\beta}(s)|&\le c\frac{\sup_{t=s}|\nabla^{2a+2b}u(t,x,y)|}
{(1+\lambda_k)^a(1+\mu_l)^b}\\
|f_{k,l,\alpha,\beta}(s)|&\le c\frac{\sup_{t=s}|\nabla^{2a+2b}Q(u)(t,x,y)|}
{(1+\lambda_k)^a(1+\mu_l)^b}
\end{align*}
Thus we get 
\begin{align*}
\|g_{k,l,\alpha,\beta}\|_\delta&\le c\frac{\|\nabla^{2a+2b}u\|_\delta}
{(1+\lambda_k)^a(1+\mu_l)^b}\\
\|f_{k,l,\alpha,\beta}\|_{2\delta}&\le c\frac{\|\nabla^{2a+2b}Q(u)\|_{2\delta}}
{(1+\lambda_k)^a(1+\mu_l)^b}
\end{align*}

From Lemma~\ref{lem:eqdiff} in Appendix~\ref{ap:ode}, we can write
$$
g_{k,l,\alpha,\beta}(t)=a_{k,l,\alpha,\beta}e^{t\lambda_{k,l,+}} +
b_{k,l,\alpha,\beta}e^{t\lambda_{k,l,-}}+ r_{k,l,\alpha,\beta}(t)
$$
with some estimates on the different terms. First we notice that
$|\lambda_{k,l,+}-\lambda_{k,l,-}|$ and $|2\delta-\Re(\lambda_{k,l,\pm})|$ are uniformly
bounded from below far from $0$ and
$\frac{(2+|\lambda_{k,l,+}|^2+|\lambda_{k,l,-}|^2)^{1/2}}
{|\lambda_{k,l,+}-\lambda_{k,l,-}|}$ is uniformly bounded. Thus there
is a uniform constant $c$ such that
\begin{align*}
\max(|a_{k,l,\alpha,\beta}|,|a_{k,l,\alpha,\beta}|)&\le c (
\|g_{k,l,\alpha,\beta}\|_\delta+ \|g'_{k,l,\alpha,\beta}\|_\delta+
\|f_{k,l,\alpha,\beta}\|_{2\delta})\\
&\le c \frac{\|\nabla^{2a+2b}u\|_\delta+ \|\nabla^{2a+2b+1}u\|_\delta+
\|\nabla^{2a+2b}Q(u)\|_{2\delta}}{(1+\lambda_k)^a(1+\mu_l)^b}
\end{align*}
and  
\begin{align*}
\|r_{k,l,\alpha,\beta}\|_{2\delta}&\le c \|f_{k,l,\alpha,\beta}\|_{2\delta}\\
&\le c \frac{\|\nabla^{2a+2b}Q(u)\|_{2\delta}}{(1+\lambda_k)^a(1+\mu_l)^b}.
\end{align*}
Besides if $\Re(\lambda_{k,l,\pm})\ge -\delta$, $t\mapsto e^{t\lambda_{k,l,\pm}}$
does not have a finite $\delta$-norm so $a_{k,l,\beta,\alpha}$ or
$b_{k,l,\alpha,\beta}$ vanishes. When $\Re(\lambda_{k,l,\pm})\le -2\delta$,
$t\mapsto e^{t\lambda_{k,l,\pm}}$ has a finite $2\delta$-norm equal to $1$.

Finally we have the following writing
\begin{align*}
u&=\sum_{-2\delta\le\Re(\lambda_{k,l,+})\le -\delta} a_{k,l,\alpha,\beta}
e^{t\lambda_{k,l,+}} \Phi_{k,\alpha}(x)\Psi_{l,\beta}(y)
+\sum_{-2\delta\le\Re(\lambda_{k,l,-})\le -\delta} b_{k,l,\alpha,\beta} e^{t\lambda_{k,l,-}}
\Phi_{k,\alpha}(x)\Psi_{l,\beta}(y)\\
&\quad+\sum_{\Re(\lambda_{k,l,+})< -2\delta} a_{k,l,\alpha,\beta}
e^{t\lambda_{k,l,+}} \Phi_{k,\alpha}(x)\Psi_{l,\beta}(y)
+\sum_{\Re(\lambda_{k,l,-})< -2\delta}b_{k,l,\alpha,\beta} e^{t\lambda_{k,l,-}}
\Phi_{k,\alpha}(x)\Psi_{l,\beta}(y)\\
&\quad+\sum r_{k,l,\alpha,\beta}(t) \Phi_{k,\alpha}(x)\Psi_{l,\beta}(y)
\end{align*}
First we notice that the first two sums are finite and are elements of the kernel of
$L$, this is the expected function $v$. Let us see that the other sums
converge and have finite $2\delta$-norms. Let $A(t,x,y)$ be the sum on
$\Re(\lambda_{k,l,+})< -2\delta$ . In
the following computation, we use the expressions of $\lambda_k$ and $\mu_l$,
their multiplicities and the $L^\infty$ estimates on $\Phi_{k,\alpha}$ and
$\Psi_{l,\beta}$.
\begin{align*}
\|A\|_{2\delta}&\le C\sum_{\Re(\lambda_{k,l,+})< -2\delta} |a_{k,l,\alpha,\beta}|
 \lambda_k^{\frac{p-1}4}\mu_l^{\frac{(n-p)-1}4}\\
&\le C\sum_{k,l,\alpha,\beta}\frac{\|\nabla^{2a+2b}u\|_\delta+
\|\nabla^{2a+2b+1}u\|_\delta+
\|\nabla^{2a+2b}Q(u)\|_{2\delta}}{(1+\lambda_k)^a(1+\mu_l)^b}
\lambda_k^{\frac{p-1}4} \mu_l^{\frac{(n-p)-1}4}\\
&\le C(\|\nabla^{2a+2b}u\|_\delta+ \|\nabla^{2a+2b+1}u\|_\delta+
\|\nabla^{2a+2b}Q(u)\|_{2\delta})\sum_{k,l}\frac{(1+k^{\frac{3p}2})
(1+l^{\frac{3(n-p)}2})}{(1+k^2)^a (1+l^2)^b}\\
&<+\infty
\end{align*}
if $a$ and $b$ are chosen such that $2a-\frac{3p}2\ge 2$ and
$2b-\frac{3(n-p)}2\ge 2$. The study of the last two sums works the same.
\end{proof}

\begin{rem*}
From the proof, the function $v$ in the kernel of $L$ can be actually written as
a finite sum of terms of the form $e^{\lambda t}\Phi(x)\Psi(y)$ with
$-2\delta<\Re(\lambda)\le -\delta$.

A second remark is that the function $r$ is a solution of $L(r)+Q(v+r)=0$
with $\|Q(v,r)\|_{2\delta}<+\infty$. So elliptic estimates give that $\|\nabla^m
r\|_{2\delta}<+\infty$.
\end{rem*}


\section{Symmetries of minimal hypersurfaces asymptotic to Simons
cones}
\label{sec:sym}


Let $\Sigma$ be a minimal hypersurface of $\R^{n+2}$ which has a
Simons cone as limit cone. If $f$ is a similarity (composition of an
isometry and a homothety) of $\R^{n+2}$, $f(\Sigma)$ is also a minimal
hypersurface asymptotic to a Simons cone. The following result says that
up to similarities, there are two such hypersurfaces (at $n$ and $p$ fixed) and even
one when $n=2p$.

\begin{thmmain}
For any $n\ge 2$ and $1\le p \le n-1$, there are two minimal hypersurfaces
$\Sigma_{n,p,\pm}$ in $\R^{n+2}$ such that the following is true. If $\Sigma$ is
a minimal hypersurface of $\R^{n+2}$ with a Simons cone as limit cone,
then $\Sigma=f(\Sigma_{n,p,\pm})$ for some $p\in\{1,\dots,n-1\}$, sign $\pm$ and a similarity $f$.
Moreover $\Sigma_{2p,p,-}=\Sigma_{2p,p,+}$.
\end{thmmain}

Actually, $C_{n,p}=f(C_{n,n-p})$ for a certain isometry $f$ of $\R^{n+2}$ so
$\Sigma_{n,p,\pm}=\Sigma_{n,n-p,\mp}$.

In the case $n=2$, we have a corollary of this which comes from the proof of the
Willmore conjecture by Marques and Neves \cite{MaNe}.

\begin{cormain}
Let $\Sigma$ be a minimal hypersurface of $\R^4$ whose density at infinity is
$\theta_\infty(\Sigma)=\frac\pi2$. Then $\Sigma=f(\Sigma_{2,1,\pm})$ for a
similarity $f$.
\end{cormain}
\begin{proof}
As explained in Section \ref{sec:density}, $\theta_\infty(\Sigma)=\frac\pi2$
implies that $\Sigma$ is asymptotic to the cone over a Clifford torus so
Theorem~\ref{th:main} applies.
\end{proof}

In order to prove Theorem~\ref{th:main} we first notice that, using an isometry,
we can assume that the limit cone in $C_{n,p}$. The cone $C_{n,p}$ is invariant
by the subgroup $O_{n,p}=O_{p+1}(\R)\times O_{n-p+1}(\R)$ of $O_{n+2}(\R)$. The
following result is the main step of the proof of Theorem~\ref{th:main}. It says
that a minimal hypersurface with $C_{n,p}$ as limit cone is also invariant by
the subgroup $O_{n,p}$.

\begin{thm}\label{th:symme}
Let $\Sigma$ be a minimal hypersurface of $\R^{n+2}$ which has $C_{n,p}$ as limit
cone. Then there is $x_0\in\R^{n+2}$ such that the translated hypersurface
$\Sigma-x_0$ is invariant by $O_{p+1}(\R)\times O_{n-p+1}(\R)$.
\end{thm}

The rest of this section is devoted to the proof of this result.


\subsection{Asymptotic behaviour of $\Sigma$}


In this section we study the asymptotic behaviour of a minimal surface with
$C_{n,p}$ as limit cone.

\begin{prop}\label{prop:asymp}
Let $\Sigma$ be a minimal hypersurface of $\R^{n+2}$ with $C_{n,p}$ as limit
cone. Then there is $x_0\in\R^{n+2}$ such
that, outside a compact set, the translate $\Sigma-x_0$ can be described as the
normal graph of a function $g$ over a subdomain of $C_{n,p}$. Moreover the function $g$ can
be written $g(t,x,y)=u(t)+f(t,x,y)$ where $u$ is in the kernel of $L$ and
$\|u\|_{3/2}<+\infty$ and $\|f\|_\delta<+\infty$ for some $\delta>2$.
\end{prop}

\begin{proof}
First, we use a result of Allard and Almgren \cite{AlAl} and Simon \cite{Sim1}
which implies that outside a compact set, the hypersurface $\Sigma$ can be
described as the normal graph of a function $g$ over $C_{n,p}$ and the function
$g$ is defined on $[\underline{t},+\infty)\times\S^p\times\S^{n-p}$ and
satisfies $\|g\|_\eps<+\infty$ for some $\eps>0$. The result of Allard and
Almgren applies since all Jacobi functions on $S_{n,p}$ comes from Killing
vectorfields of $\S^{n+1}$ (see Section~6 in \cite{AlAl}). Decreasing slightly $\eps$
if necessary, we can assume that $-2\eps\neq \lambda_{k,l,\pm}$ and apply
Proposition~\ref{prop:improv}. So $g=v+r$ with $v$ in the kernel of $L$ with
decay between $-\eps$ and $-2\eps$ and $\|r\|_{2\eps}<+\infty$. If there is no
element in the kernel of $L$ with decay between $-\eps$ and $-2\eps$, we get
$\|g\|_{2\eps}<\infty$; in that case we have then improved the decay of $g$.
So we can iterate this argument until we get a first non vanishing element in
the kernel.

The first decay of elements in the kernel is given by
$\lambda_{1,0,+}=\lambda_{0,1,+}=-1$. Besides, the eigenfunctions
$\Phi_{1,\alpha}$ and $\Psi_{1,\beta}$ are the coordinates functions so $g$ can
be written
$$
g(t,x,y)=e^{-t}(a_1x_1+\cdots+a_{p+1}x_{p+1}+b_1y_1+\cdots+b_{n-p+1}y_{n-p+1})+
r(t,x,y)
$$ 
with $\|r\|_{1+\eps}<+\infty$ for some $\eps>0$. This can also be written
$$
g(t,x,y)=e^{-t}(X_0, N(t,x,y))+r(t,x,y)
$$

The first term can be interpreted as a translation. More precisely, in the
parametrization $Y$, a term $(X_0,N)N$ appears. So the translated hypersurface
$\Sigma-X_0$ can be expressed as the normal graph of a function $w$ over
$C_{n,p}$ with the following estimates $\|w\|_{1+\eps}<+\infty$ for some
$\eps>0$.

From now on, we study the asymptotic behaviour of $\Sigma-X_0$ as a normal graph
over $C_{n,p}$. We still call $\Sigma$ this translated hypersurface.

If we apply Proposition~\ref{prop:improv}, we get the following writing $w=v+r$
with $v$ in the kernel of $L$ with a decay between $-1-\eps$ and
$-2-2\eps$ and $\|r\|_{2+2\eps}<+\infty$. If $n>3$, all $\lambda_{k,l,\pm}$ are
outside the segment $[-2-2\eps,-1-\eps]$ for $\eps$ close enough to $0$, so $v$
is vanishing and the proposition is proved. If $n=3$, $\lambda_{0,0,\pm}=-2\pm
i\sqrt 2$ is the only possibilities. This value comes from the constant
functions on $\S^p$ and $\S^{n-p}$ so $v$ only depends on $t$, so the
proposition is proved.

When $n=2$, we have two possibilities, $\lambda_{0,0,\pm}=-\frac 32\pm
i\frac{\sqrt 7}2$ coming from constant functions on $\S^p$ and $\S^{n-p}$ and $\lambda_{1,0,-}=\lambda_{0,1,-}=-2$ from the coordinate functions. So $w$ can be written
$$
w(t,x,y)=ae^{-\frac{3t}2}\cos(\frac{\sqrt7t}2+\phi_0)+e^{-2t}(X_1,N(t,x,y))+
r(t,x,y).
$$
Let us prove that actually $X_1$ is vanishing. To prove this, we use a flux
argument. Let us recall that if $\Ome$ is a subset of $\Sigma$ with smooth
boundary and $\nu$ denote the normal to $\partial \Ome$ tangent to $\Sigma$,
then the flux of $\nu$ across $\partial\Ome$ vanishes; more precisely:
$$
\int_{\partial\Ome}\nu=0
$$
We apply this result to the bounded subset $\Ome_{t_0}$ of $\Sigma$ whose
boundary is the hypersurface $\{t=t_0\}$. Using the above expression of $w$ in
Appendix~\ref{ap:estim}, we estimate this flux (see Equation~\eqref{eq:estim})
and we get
$$
0=\int_{\partial\Ome_{t_0}}\nu=\int_{\S^1\times\S^1}\frac12(X_1,N)N+
O(e^{-{2\eps}t}).
$$
Taking the limit $t\rightarrow +\infty$ and taking the scalar product with
$X_1$, we get that $(X_1,N)=0$ for all $(x,y)\in\S^1\times\S^1$ : so $X_1=0$.
This finishes the proof of the proposition.
\end{proof}

We recall that the derivatives of $f$ also have finite $\delta$-norms.


\subsection{Alexandrov reflection}


Let $\Sigma$ be a minimal hypersurface in $\R^{n+2}$ with $C_{n,p}$ as limit
cone. We translate $\Sigma$ such that the asymptotic behaviour of
Proposition~\ref{prop:asymp} is true (the translated hypersurface is still
named $\Sigma$). In this section, we use this asymptotic behaviour to prove
that $\Sigma$ is
invariant by $O_{p+1}(\R)\times O_{n-p+1}(\R)$ and then prove
Theorem~\ref{th:symme}. 

Let us denote the coordinates of $\R^{n+2}$ by
$(x_1,\cdots,x_{p+1},y_1,\cdots,y_{n-p+1})$. Actually, we
are going to prove that $\Sigma$ is  symmetric with respect to
$\{x_1=0\}$. If $s\in O_{p+1}(\R)$, $s(\Sigma)$ satisfies the same hypotheses as
$\Sigma$ so $s(\Sigma)$ will be symmetric with respect to $\{x_1=0\}$ and then
$\Sigma$ will be symmetric with respect to $s^{-1}(\{x_1=0\})$. All these
symmetries imply that $\Sigma$ is $O_{p+1}(\R)$-invariant. For the
$O_{n-p+1}(\R)$-invariance, the proof is similar by exchanging $p$ by $n-p$.

Outside a compact set, the hypersurface $\Sigma$ is the normal graph a function $g$
that can be written as in Proposition~\ref{prop:asymp}
$g(t,x,y)=f(t)+O(e^{-\delta t})$ with $\delta>2$ and $f=O(e^{-\frac{3t}2})$. The first
coordinate of
the point $Y(t,x,y)$ is given by $e^t(\sqrt{\frac pn}+
g(t,x,y)\sqrt{\frac{n-p}n}) x_1$. In the following we are interested to the
following subset of $\Sigma$:
$$
\Sigma_{t_0,a}=Y(\{(t,x,y)\in \R\times\S^p\times\S^{n-p}\,|\,t\ge t_0,
e^t(\sqrt{\frac pn}+ g(t,x,y)\sqrt{\frac{n-p}n}) x_1> a\}).
$$
So a point of $\Sigma$ is in $\Sigma_{t_0,a}$ if it is sufficiently far from the
origin and its first coordinate is larger than $a$.

We denote by $\pi$ the projection map of $\R^{p+1}$ on $\{x_1=0\}$. We have a
first lemma that describes $\Sigma_{t_0,a}$. 

\begin{lem}\label{lem:graph}
There are $t_0$ and $c>0$ such that for any $a>0$ the map
$(\pi,\id):\R^{n+2}\rightarrow\{x_1=0\}\times\R^{n-p+1}$ is injective on
$\Sigma_{t_a,a}$ where
$$
t_a=\max (t_0, \ln \frac ca)
$$
\end{lem}
\begin{proof}
$t_0$ will be chosen sufficiently large so that $e^{-t_0}$ is small enough with respect to
quantities appearing in the asymptotic behaviour of $g$. We denote $x=(x_1,
\pi(x))$. If the map is not
injective, we have $(t,x,y)$ and $(t',x',y')$ ($t'\ge t$) such that
\begin{gather}\label{eq:proj1}
e^t(\sqrt{\frac pn}+g\sqrt{\frac{n-p}n})\pi(x)= e^{t'}(\sqrt{\frac
pn}+g'\sqrt{\frac{n-p}n})\pi(x')\\ \label{eq:proj2}
e^t(\sqrt{\frac {n-p}n}-g\sqrt{\frac pn})y= e^{t'}(\sqrt{\frac {n-p}n}-g'\sqrt{\frac pn})y'
\end{gather}
with $g=g(t,x,y)$ and $g'=g(t',x',y')$.

From \eqref{eq:proj2}, $y=y'$ and $e^t(\sqrt{\frac {n-p}n}-g\sqrt{\frac pn})=
e^{t'}(\sqrt{\frac {n-p}n}-g'\sqrt{\frac pn})$. So if $h=e^t(\sqrt{\frac
{n-p}n}-f(t)\sqrt{\frac pn})$ and $h'=e^{t'}(\sqrt{\frac
{n-p}n}-f(t')\sqrt{\frac pn})$, we get $h'-h=O(e^{(1-\delta) t})(|t-t'|+|x-x'|)$. We have
$\partial_t h=e^t(\sqrt{\frac{n-p}n}+O(e^{-\frac{3t}2}))\ge \frac1n e^t$ if $t\ge t_0$
large.

So $e^t(t'-t)\le O(e^{(1-\delta)t})(|t-t'|+|x-x'|)$. Then
$$
|t'-t|\le ce^{-\delta t}|x-x'|
$$
for $t'\ge t\ge t_0$ large.

Thus 
$$
e^{t'}(\sqrt{\frac pn}+g'\sqrt{\frac{n-p}n})=e^t(\sqrt{\frac
pn}+g\sqrt{\frac{n-p}n})+ O(e^{(1-\delta)t})|x-x'|
$$
Using this in \eqref{eq:proj1}, we get
$$
e^t\sqrt{\frac pn}|\pi(x-x')|=O(e^{(1-\delta)t})|x-x'|
$$
On the hemisphere $\S^p\cap \{x_1>0\}$, we have $|\pi(x-x')|\ge
\frac{\min(x_1,x_1')}{\sqrt2} |x-x'|$. This implies $e^t\sqrt{\frac
pn}\min(x_1,x_1')|x-x'|=O(e^{(1-\delta)t})|x-x'|$. Thus
$$
a |x-x'|\le ce^{(1-\delta)t}|x-x'|
$$
for $t\ge t_0$ large. Since $\delta>2$ it implies $x=x'$ and then $t=t'$ if 
$$
t\ge\max(t_0,\ln\frac ca)>\frac 1{\delta-1}\ln \frac ca
$$
\end{proof}

This lemma implies that large parts of $\Sigma$ can be described as graphs in the
$x_1$ direction. For $a>0$, we denote by $S_a$ the symmetry with respect to $x_1=a$.

\begin{lem}\label{lem:sym}
There are constants $t_0$, $b>0$ and $c>0$ such that for any $a>0$ the image of
$\Sigma_{t_a,a}$ by $S_a$ does not intersect $\Sigma$ where
$$
t_a=\max (t_0, b\ln \frac ca)
$$
\end{lem}

\begin{proof}
As above, $t_0$ will be chosen such that $e^{-t_0}$ is sufficiently small with respect to
quantities appearing in the asymptotic behaviour of $g$. We have
$$S_a(Y(t,x,y))=\begin{pmatrix}
2a-e^t(\sqrt{\frac pn}+g\sqrt{\frac{n-p}n})x_1\\
e^t(\sqrt{\frac pn}+g\sqrt{\frac{n-p}n})\pi(x)\\
e^t(\sqrt{\frac{n-p}n}-g\sqrt{\frac pn})y
\end{pmatrix}
$$
So $|S_a(Y((t,x,y))|\ge e^t(\sqrt{\frac{n-p}n}-g\sqrt{\frac pn})\ge c e^t$ if $t\ge t_0$
large. Thus
$S_a((Y(t,x,y))$ is outside a large ball if $t>t_0$ is large. So we can care
only about the part of $\Sigma$ which is parametrized by the normal graph and
with large $t$: if $S_a(Y(t,x,y))$ is inside $\Sigma$, this point can be written
$Y(t',x',y')$ with $t'\ge t_O$ if $t$ if large.

If $a\le e^t(\sqrt{\frac pn}+g\sqrt{\frac{n-p}n})x_1\le 3a/2$, it is clear that
$S_a(Y(t,x,y))$ is not in $\Sigma$ because of Lemma~\ref{lem:graph} applied with
$a/2$ in place of $a$.

Now we assume that $e^t(\sqrt{\frac pn}+g\sqrt{\frac{n-p}n})x_1\ge 3a/2$ and we have
\begin{equation}\label{eq:Sa=Y}
S_a(Y(t,x,y))=Y(t',x',y')
\end{equation}
For $(\alpha,\beta)\in\R^{p+1}\times\R^{n-p+1}$, let
$Q(\alpha,\beta)=(n-p)|\alpha|^2-p|\beta|^2$. We have
$Q(S_a(Y(t,x,y)))=Q(Y(t',x',y'))$ thus
\begin{align*}
e^{2t'}(2g'\sqrt{p(n-p)}+g'^2(n-2p))= &(n-p)4a(a-e^t(\sqrt{\frac pn}
+g\sqrt{\frac{n-p}n})x_1)\\
&+e^{2t}(2g\sqrt{p(n-p)}+g^2(n-2p))
\end{align*}
Using $e^t(\sqrt{\frac pn}+g\sqrt{\frac{n-p}n})x_1\ge 3a/2$, this gives
\begin{equation}\label{eq:truc}
e^{2t'}(2g'\sqrt{p(n-p)}+g'^2(n-2p)) -
e^{2t}(2g\sqrt{p(n-p)}+g^2(n-2p))\le -(n-p)2a^2
\end{equation}
From \eqref{eq:Sa=Y}, we also have
$$
e^t(\sqrt{\frac{n-p}n}-g\sqrt{\frac pn})y=
e^{t'}(\sqrt{\frac{n-p}n}-g'\sqrt{\frac pn})y'
$$
As in the Lemma~\ref{lem:graph}, this gives $|t'-t|\le ce^{-\delta t}$ if $t\ge t_0$
large. Using this in \eqref{eq:truc}, we finally get
$$
ce^{(2-\delta)t}\ge (n-p)2a^2
$$
Lemma~\ref{lem:sym} is then proved since $\delta>2$.
\end{proof}

Now we can apply the Alexandrov reflection procedure to prove the following
result.

\begin{lem}
The surface $\Sigma$ is symmetric with respect to $\{x_1=0\}$.
\end{lem}

\begin{proof}
First we denote by $\Sigma_a=\Sigma\cap\{x_1>a\}$. Let also $t_a$ be given by
Lemma~\ref{lem:sym}. If $a>0$ is large, $\Sigma_a$ is a subset of the part of
$\Sigma$ which is a normal graph. Besides $|Y(t,x,y)|\ge a$ so $t$ is large on
$\Sigma_a$ if $a$ is large. This implies that for $a$ sufficiently large
$\Sigma_a=\Sigma_{t_a,a}$. So from Lemma~\ref{lem:sym}, $S_a(\Sigma_a)\cap
\Sigma=\emptyset$ for $a$ large.

For any $a>0$, $\Sigma_a\setminus \Sigma_{t_a,a}$ is a bounded subset, so if
there is some $a'>0$ such that $S_{a'}(\Sigma_{a'})\cap \Sigma \neq\emptyset$,
there is a first contact point between $S_a(\Sigma_a)$ and $\Sigma$. There is
$a_0>0$ such that one of the following two possibilities occurs: there is $p_0\in
\Sigma\cap S_{a_0}(\Sigma_{a_0})$ such that $S_{a_0}(\Sigma_{a_0})$ lies on one side of
$\Sigma$ near $p_0$ or there is $p_0\in \partial \Sigma_{a_0}$ such that $\Sigma$ is
normal to $\{x_1=a_0\}$ at $p_0$ and $S_{a_0}(\Sigma_{a_0})$ lies on one side of $\Sigma$
near $p_0$. In both cases, $\Sigma$ and $S_{a_0}(\Sigma_{a_0})$ can be described near $p_0$
as graphs over $T_{p_0}\Sigma$. In the first case, applying the maximum principle at $p_0$,
we get $S_{a_0}(\Sigma_{a_0})\subset\Sigma$ which is not possible by
Lemma~\ref{lem:sym}. In the second case, the boundary maximum principle can be applied at
the boundary point $p_0$ to get the same contradiction (see \cite{Ale}).

This implies that $\Sigma_0$ is a graph in the $x_1$ direction and
$\Sigma\cap\{x_1<0\}$ lies on one side of $S_0(\Sigma_0)$ in $\{x_1<0\}$. We
notice that $S_0(\Sigma)$ has the same asymptotic behaviour as $\Sigma$. Thus,
applying the same argument to $S_0(\Sigma)$, we get that $\Sigma\cap\{x_1<0\}$
is also a graph in the $x_1$ direction. Now because of the asymptotic behaviour
of $\Sigma$ the first coordinate of the normal to $\Sigma$ changes its sign. So
$\Sigma$ is normal to $\{x_1=0\}$ and the maximum principle implies that
$S_0(\Sigma)=\Sigma$.
\end{proof}


\section{Minimal hypersurfaces invariant by $\Onp$}
\label{sec:ode}


In order to finish the proof of Theorem~\ref{th:main}, we need to understand all
the minimal hypersurfaces that are invariant by $\Onp$. This study has been
partially done by Bombieri, de Giorgi and Giusti in \cite{BoGiGi} in the case
$n=2p\ge 6$. It has been completed by Alencar, Barros, Palmas, Reyes and Santos
\cite{ABPRS}; here for sake of completeness, we write the part of the study
which is necessary for our result. We want to prove that up to homotheties there
is two minimal hypersurfaces invariant by $\Onp$ with $C_{n,p}$ as limit cone.

So we are looking for a pair of functions $a$, $b$ defined on an interval $I$
such that the hypersurface parametrized by
$$
X:I\times \S^p\times\S^{n-p}\longrightarrow \R^{n+2}; (t,x,y)\longmapsto
(a(t)x,b(t)y)
$$
is minimal.

The hypersurface is minimal if $a$ and $b$ satisfy to a certain ordinary
differential equation:
$$
0=a''b'-b''a'+(a'^2+b'^2)\Big((n-p)\frac{a'}b-p\frac{b'}a\Big)
$$
Since being minimal is invariant by homotheties, $(\lambda a,\lambda b)$ is a
solution if $(a,b)$ is a solution. In order to use this property we introduce
new parameters by these expressions:
\begin{align*}
(a,b)&=e^\rho(\cos\theta,\sin\theta)\\
(a',b')&=e^r(\cos\phi,\sin\phi)
\end{align*}
The above ode is then equivalent to
$$
\begin{cases}
\rho'&=e^{r-\rho}\cos(\theta-\phi)\\
\theta'&=-e^{r-\rho}\sin(\theta-\phi)\\
\phi'&=e^{r-\rho}\frac{(n-2p)\cos(\theta-\phi)+n\cos(\theta+\phi)}{\sin2\theta}
\end{cases}
$$
So, changing the time parameter, we get the following system
$$
\begin{cases}
\rho'&=\sin2\theta\cos(\theta-\phi)\\
\theta'&=-\sin2\theta\sin(\theta-\phi)\\
\phi'&=(n-2p)\cos(\theta-\phi)+n\cos(\theta+\phi)
\end{cases}
$$
So we are let to understand the flow lines of
$$
(\theta,\phi)'=Y(\theta,\phi)=(-\sin2\theta\sin(\theta-\phi),(n-2p)
\cos(\theta-\phi)+n\cos(\theta+\phi))
$$
We denote by $Y_1$ and $Y_2$ the two components of the vectorfield. First we
remark that $Y_1(k\pi/2,\phi)=0$ so the subsets $\{k\pi/2\le\theta\le(k+1)\pi/2\}$
are stable. Moreover, we have $Y(\theta+\pi,\phi)=-Y(\theta,\phi)$,
$Y(\theta,\phi+\pi)=-Y(\theta,\phi)$ and $Y(-\theta,-\phi)=Y(\theta,\phi)$. So
we need to understand the vectorfield on $[0,\pi/2]\times(-\pi/2,\pi/2]$.

In this subset, the singular points are the following 
\begin{itemize}
\item a saddle point $(\pi/2,0)$ with stable direction $(0,1)$ and unstable one
$(p+1,p-n)$,
\item a saddle point $(0,\pi/2)$ with stable direction $(0,1)$ and unstable one
$(n+1-p,-p)$ and
\item a stable nodal or focal point $(\theta_0,\theta_0)$ where
$\theta_0\in(0,\pi/2)$ satisfies $\cos\theta_0=\sqrt{\frac pn}$ (if $n\le 6$,
the roots of $dY(\theta_0,\theta_0)$ are conjugate complex numbers with negative
real parts and, if $n\ge 7$ the roots are negative real numbers).
\end{itemize}

The properties of the vectorfield $Y$ are summarized in the following
proposition (see also Figure~\ref{fig:vectorfield}).

\begin{prop}
The vectorfield $Y$ satisfies to the following properties :
\begin{itemize}
\item if $\phi\in(-\pi/2,0)$, $Y_2(\theta,\phi)> 0$ for all $\theta\in
[0,\pi/2]$ and
\item there is a continuous decreasing surjective function
$\tau:[0,\theta_0]\rightarrow[0,\pi/2]$ such that $\phi \mapsto \phi+\tau(\phi)$ decreases
and $Y$ points inside $[\phi,\phi+\tau(\phi)]^2$ along its boundary for
$\phi\in[0,\theta_0)$. 
\end{itemize}
\end{prop}
\begin{proof}
We have $Y_2(\theta,\phi)=2(n-p)\cos\phi\cos\theta-2p\sin\phi\sin\theta$, so the first
property is clear.

For the second property, we first notice that, for $\theta,\phi\in[0,\pi/2]$,
$Y_1(\theta,\phi)$ has the same sign as $\phi-\theta$. We remark also that
$Y_2(\theta,\phi)=Y_2(\phi,\theta)$. Moreover 
\begin{align*}
Y_2(\phi+\tau,\phi)&=( n-2p +n\cos2\phi)\cos \tau-n\sin 2\phi\sin \tau\\
&=n(\cos2\phi-\cos 2\theta_0)\cos \tau-n\sin2\phi\sin \tau
\end{align*}
So, for $\phi\in[0,\theta_0]$, $\tau\mapsto Y_2(\phi+\tau,\phi)$ is non increasing for
$\tau\in[0,\pi/2]$. It vanishes for $\tau=\tau(\phi)=\arctan(\frac{\cos
2\phi-\cos2\theta_0}{\sin2\phi})$. Thus it is non negative for $\tau\in[0,\tau(\phi)]$.
When $\psi\ge \theta_0$, $\tau\mapsto Y_2(\psi+\tau,\psi)$ is non increasing for
$\tau\in[-\pi/2,0]$. Since $Y_2(\phi,\phi+\tau(\phi))=Y_2(\phi+t(\phi),\phi)=0$, it
implies that $Y_2(\phi+t(\phi)+\tau,\phi+\tau(\phi))$ is non positive for
$\tau\in[-\tau(\phi),0]$. The fact that $\phi\mapsto \phi+\tau(\phi)$ is decreasing is just
a computation. This finishes the proof of the second item.
\end{proof}

We notice that $\theta_0+\tau(\theta_0)=\theta_0$ so $\phi+\tau(\phi)\ge \theta_0$ for any
$\phi\in[0,\theta_0]$.

The above properties are sufficient to describe all the integral curves of $Y$
passing trough a point in $(\theta,\phi)\in(0,\pi/2)\times(-\pi/2,\pi/2)$. We
have four possibilities :
\begin{itemize}
\item an integral curve starting from $(\theta_0,\theta_0-\pi)$ and ending at
$(\theta_0,\theta_0)$,
\item an integral curve starting from $(\theta_0,\theta_0+\pi)$ and ending at
$(\theta_0,\theta_0)$,
\item the unstable manifold starting from $(\pi/2,0)$ and ending at
$(\theta_0,\theta_0)$ or
\item the unstable manifold starting from $(0,\pi/2)$ and ending at
$(\theta_0,\theta_0)$.
\end{itemize}

The behaviour of $\theta$ and $\phi$ along the unstable manifolds close to
$(\pi/2,0)$ and $(0,\pi/2)$ implies that $\rho$ has a limit when the time
parameter goes to $-\infty$. This implies that these two integral curves
generate minimal hypersurfaces that extend smoothly near this endpoint.

The behaviour of $\theta$ and $\phi$ near $(\theta_0,\theta_0)$ (and also
$(\theta_0,\theta_0\pm\pi)$) implies that $\rho$ grows linearly when the time
parameter is close to $\pm\infty$. This implies that all these integral curves
generate proper minimal hypersurfaces whose asymptotic behaviour is given by
twice the cone $C_{n,p}$ in the first two cases and once the cone $C_{n,p}$ in
the last two cases.

Since we study minimal hypersurfaces asymptotic to once the cone $C_{n,p}$,
there is only two possibilities that correspond to the two unstable manifolds
(when $n=2p$, extra symmetries of $Y$ implies that the two integral curves
are symmetric to each other). Moreover, we know that the density at infinity of
$C_{n,p}$ is less than $2$ so these two hypersurfaces are embedded. These two
hypersurfaces are precisely the hypersurfaces $\Sigma_{n,p,\pm}$ that appear in the
statement of Theorem~\ref{th:main}. 

\begin{figure}[h]
\begin{center}
\resizebox{0.5\linewidth}{!}{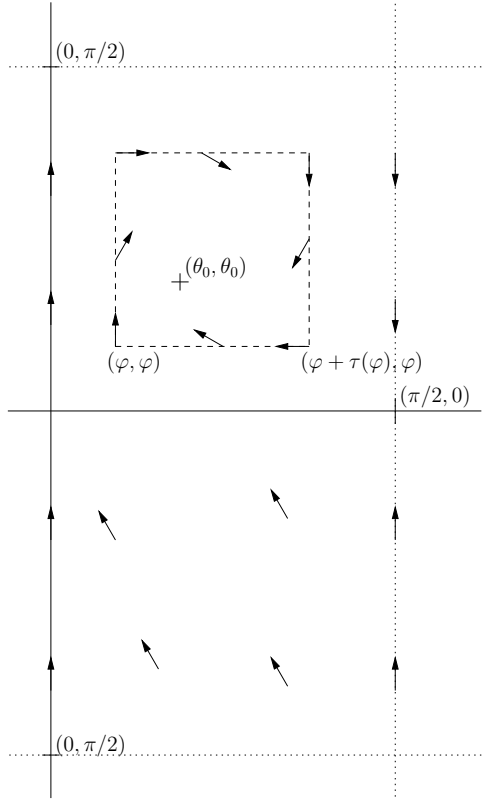}
\caption{The vector field $Y$ in $[0,\pi/2]\times[-\pi/2,\pi/2]$
\label{fig:vectorfield}}
\end{center}
\end{figure}


\appendix


\section{An ODE lemma}
\label{ap:ode}


In this appendix we prove the following lemma about solutions of linear ode.

\begin{lem}\label{lem:eqdiff}
Let $\lambda,\mu\in\C$ with $\lambda\neq \mu$. Let $g$ and $f$ be two smooth functions on
$\R_+$ such that
$$
g''-(\lambda+\mu)g'+\lambda\mu g=f
$$
We assume that $\|f\|_\delta$ is finite where $\delta\neq -\Re(\lambda),-\Re(\mu)$. Then $g$
can be written
$g(t)=ae^{\lambda t}+be^{\mu t}+v(t)$ with the following estimates:
\begin{align*}
\max(|a|,|b|)&\le c\frac{(2+|\lambda|^2+|\mu|^2)^{1/2}}{|\lambda-\mu|}
(|g(0)|+|g'(0)|)\\
&\quad\quad\quad+2\frac{\max(0,(\delta+\Re(\lambda))^{-1},(\delta+
\Re(\mu))^{-1})}{|\lambda-\mu|}\|f\|_\delta\\
\|v\|_\delta&\le \frac{|\delta+\Re(\lambda)|^{-1}+|\delta+
\Re(\mu)|^{-1}}{|\lambda-\mu|}\|f\|_\delta
\end{align*}
for some universal constant $c>0$.
\end{lem}

\begin{proof}
As a solution of such an ode, $g$ can be written
$$
g(t)=ae^{\lambda t}+be^{\mu t}+\frac1{\lambda-\mu}\int_0^tf(u)(e^{\lambda(t-u)}-
e^{\mu(t-u)})du $$
with $a$ and $b$ solution of 
$$
\begin{cases}
g(0)&=a+b\\
g'(0)&=\lambda a+\mu b
\end{cases}.
$$
So $\begin{pmatrix}a\\ b\end{pmatrix}=\frac1{\mu-\lambda}
\begin{pmatrix}\mu&-1\\ \lambda&1 \end{pmatrix} \begin{pmatrix}g(0)\\
g'(0)\end{pmatrix}$ so 
$$
\max(|a|,|b|)\le
c\frac{(2+|\lambda|^2+|\mu|^2)^{1/2}}{|\lambda-\mu|}(|g(0)|+|g'(0)|).
$$
If $\delta+\Re(\lambda)>0$
\begin{align*}
\int_0^tf(u)e^{-\lambda u}du&=\int_0^{+\infty}f(u)e^{-\lambda u}du+\int_{+\infty}^t
f(u)e^{-\lambda u}du\\
&=A+\int_{+\infty}^tf(u)e^{-\lambda u}du
\end{align*}
with $|A|\le \int_0^{+\infty}\|f\|_\delta e^{-(\delta+\Re(\lambda))u}du\le
\frac{\|f\|_\delta}{\delta+\Re(\lambda)}$ and
$$
\Big|\int_{+\infty}^t f(u)e^{-\lambda u}du\Big|\le
\frac{\|f\|_\delta}{\delta+\Re(\lambda)}e^{-(\delta+\Re(\lambda))t}.
$$
If $\delta+\Re(\lambda)<0$, we have:
$$
\Big|\int_0^t f(u)e^{-\lambda u}du\Big|\le
\frac{\|f\|_\delta}{-(\delta+\Re(\lambda))}e^{-(\delta+\Re(\lambda))t}.
$$
This finally gives $g=ae^{\lambda t}+be^{\mu t}+v$ with the expected estimates.
\end{proof}


\section{A flux computation}
\label{ap:estim}


In this appendix, we make the computation of the flux used in the proof of
Proposition~\ref{prop:asymp}. So we use some notation introduced in this proof.

We are in the case $n=2$, so the hypersurface $\Sigma$ is parametrized by
$$
Y(t,\theta,\phi)=e^t\big(R(\theta,\phi)+w(t,\theta,\phi)N(\theta,\phi)\big)
$$
where
$$
R(\theta,\phi)=\frac1{\sqrt2}\begin{pmatrix}\cos\theta\\ \sin\theta\\ \cos\phi\\
\sin\phi\end{pmatrix}\quad\textrm{ and }\quad N(\theta,\phi)=
\frac1{\sqrt2}\begin{pmatrix}\cos\theta\\ \sin\theta\\ -\cos\phi\\
-\sin\phi\end{pmatrix}
$$
Moreover, we notice that $w$ and $w_t$ are $O(e^{-\frac{3t}2})$ and $w_\theta$
and $w_\phi$ are $O(e^{-2t})$. We also define
$$
E_\theta(\theta,\phi)=\begin{pmatrix}-\sin\theta\\ \cos\theta\\
0\\0\end{pmatrix}\quad\textrm{ and }\quad E_\phi(\theta,\phi)=
\begin{pmatrix}0\\ 0\\ -\sin\phi\\ \cos\phi\end{pmatrix}
$$
We notice that $R,E_\theta,E_\phi,N$ is an oriented orthonormal basis.
We have
\begin{align*}
Y_t&=e^t\big(R+(w+w_t)N\big)\\
Y_\theta&=e^t\big(\frac{1+w}{\sqrt2}E_\theta+w_\theta N\big)\\
Y_\phi&=e^t\big(\frac{1-w}{\sqrt2}E_\phi+w_\phi N\big)
\end{align*}

So the cross product of $X_t$, $X_\theta$ and $X_\phi$ is
\begin{align*}
\bigwedge(Y_t,Y_\theta,Y_\phi)&=e^{3t}(\frac{1-w^2}2N-\frac{1-w}{\sqrt2}w_\theta
E_\theta-\frac{1+w}{\sqrt2}w_\phi E_\phi-\frac{1-w^2}2(w+w_t)R)\\
&=e^{3t}(\frac12N-\frac{w+w_t}2R-\frac{w_\theta}{\sqrt2}E_\theta-
\frac{w_\phi}{\sqrt2}E_\phi+O(e^{-\frac{5t}2}))
\end{align*}
So the unit normal $n(t,\theta,\phi)$ to the graph has the following expression
$$
n=N-(w+w_t)R-\sqrt2 w_\theta E_\theta-\sqrt2 w_\phi E_\phi+O(e^{-\frac{5t}2})
$$ 
Then to get an expression of the normal $\nu$ to the boundary of $\Ome_t$, we compute
$$
\bigwedge(n,Y_\theta,Y_\phi)=e^{2t}(-\frac12R-\frac{w+w_t}2N+O(e^{-\frac{5t}2}))
$$
So $\nu=-R-(w+w_t)N+O(e^{-\frac{5t}2})$. Besides the surface element along
$\partial\Ome_t$ can be estimated by $e^{2t}(\frac12+O(e^{-3t}))d\theta d\phi$.
So the flux $F$ of $\nu$ is given by
\begin{align*}
F&=e^{2t}\int_0^{2\pi}\int_0^{2\pi}(-R-(w+w_t)N+O(e^{-\frac{5t}2}))(\frac12+
O(e^{-3t}))d\theta d\phi\\
&=\frac{e^{2t}}2\int_0^{2\pi}\int_0^{2\pi}-R-(w+w_t)N+O(e^{-\frac{5t}2}) d\theta
d\phi
\end{align*}
We notice that the integral of $R$ and $N$ vanishes, so because of the
expression of $w$ we get the following estimates
\begin{equation}\label{eq:estim}
\begin{split}
F&=\frac{e^{2t}}2\int_0^{2\pi}\int_0^{2\pi}e^{-2t}(X_1,N)N+O(e^{-(2+2\eps)t})
d\theta d\phi\\
&=\int_0^{2\pi}\int_0^{2\pi}\frac12(X_1,N)N d\theta d\phi
+O(e^{-2\eps t})
\end{split}
\end{equation}

\bibliographystyle{plain}
\bibliography{../reference.bib}

\noindent \textsc{Laurent Mazet, Universit\'e Paris-Est, LAMA (UMR 8050), UPEC,
UPEM, CNRS, 61, avenue du G\'en\'eral de Gaulle, F-94010 Cr\'eteil cedex,
France}

\noindent \texttt{laurent.mazet@math.cnrs.fr}

\end{document}